\documentclass[12pt,reqno]{amsart}
\usepackage{amsmath,amsthm,amscd,amsfonts,amssymb,color}
\usepackage{cite}
\usepackage[mathscr]{eucal}
\usepackage{mathtools}

\usepackage[bookmarksnumbered,colorlinks,plainpages]{hyperref}
\newtheorem{theorem}{Theorem}[section]
\newtheorem{lemma}[theorem]{Lemma}

\theoremstyle{definition}
\newtheorem{definition}[theorem]{Definition}

\newtheorem{Open Prob}[theorem]{Open Problem}
\theoremstyle{remark}

\numberwithin{equation}{section}
\def\DJ{\leavevmode\setbox0=\hbox{D}\kern0pt\rlap{\kern.04em\raise.188\ht0\hbox{-}}D}
\begin{document}

\title[$p$-Proximal contraction]{A remark on the paper ``A note on the paper Best proximity point results for $p$-proximal contractions"}

\author[S.\ Som]
{Sumit Som}

\address{           Sumit Som,
                    Department of Mathematics,
                    School of Basic and Applied Sciences, Adamas University, Barasat-700126, India.}
                    \email{somkakdwip@gmail.com}

\subjclass {$54H25$, $47H10$}

\keywords{Best proximity point, $p$-proximal contraction, Banach contraction principle}

\begin{abstract}
%Let $X$ be a topological space and $g$ be a real valued continuous function defined on the product space $X\times X.$ 
Recently, In the year 2020, Altun et al. \cite{AL} introduced the notion of $p$-proximal contractions and discussed about best proximity point results for this class of mappings. Then in the year 2021, Gabeleh and Markin \cite{GB} showed that the best proximity point theorem proved by Altun et al. in \cite{AL} follows from the fixed point theory. In this short note, we show that if the $p$-proximal contraction constant $k<\frac{1}{3}$ then the existence of best proximity point for $p$-proximal contractions follows from the celebrated Banach contraction principle.
\end{abstract}

\maketitle

\section{\bf{Introduction}}
Metric fixed point theory is an essential part of Mathematics as it gives sufficient conditions which will ensure the existence of solutions of the equation $F(x)=x$ where $F$ is a self mapping defined on a metric space $(M,d).$ Banach contraction principle for standard metric spaces is one of the important results in metric fixed point theory and it has lot of applications. Let $A,B$ be non-empty subsets of a metric space $(M,d)$ and $Q:A\rightarrow B$ be a non-self mapping. A necessary condition, to guarantee the existence of solutions of the equation $Qx=x,$ is $Q(A)\cap A\neq \phi.$ If $Q(A)\cap A= \phi$ then the mapping $Q$ has no fixed points. In this case, one seek for an element in the domain space whose distance from its image is minimum i.e, one interesting problem is to $\mbox{minimize}~d(x,Qx)$ such that $x\in A.$ Since $d(x,Qx)\geq d(A,B)=\inf~\{d(x,y):x\in A, y\in B\},$ so, one search for an element $x\in A$ such that $d(x,Qx)= d(A,B).$ Best proximity point problems deal with this situation. Authors usually discover best proximity point theorems to generalize fixed point theorems in metric spaces. Recently, In the year 2020, Altun et al. \cite{AL} introduced the notion of $p$-proximal contractions and discussed about best proximity point results for this class of mappings. Then, in the year 2021, Gabeleh and Markin \cite{GB} showed that the best proximity point theorem proved by Altun et al. in \cite{AL} follows from a result in fixed point theory. In this short note, we show that if the $p$-proximal contraction constant $k<\frac{1}{3}$ then the existence of best proximity point for $p$-proximal contractions follows from the Banach contraction principle.

\section{\bf{Main results}}

We first recall the following definition of $p$-proximal contraction from \cite{AL}.

\begin{definition}\cite{AL}
Let $(A, B)$ be a pair of nonempty subsets of a metric space $(M,d).$ A mapping $f:A\rightarrow B$ is said to be a $p$-proximal contraction if there exists $k\in (0,1)$ such that   \[
\begin{rcases}
 d(u_1, f(x_1))= d(A, B)\\
 d(u_2, f(x_2))= d(A, B)
 \end{rcases}
  {\Longrightarrow d(u_1,u_2)\leq k \Big(d(x_1,x_2)+|d(u_1,x_1)-d(u_2,x_2)|\Big)}
  \] for all $u_1, u_2, x_1, x_2 \in A,$ where $d(A,B) =  \inf\Big\{d(x, y): x\in A,\ y\in B\Big\}.$
\end{definition}
In this paper, we call the constant $k$ in the above definition as $p$-proximal contraction constant. The following notations will be needed.
Let $(M,d)$ be a metric space and $A,B$ be nonempty subsets of $M.$ Then
$$A_0=\{x\in A: d(x,y)=d(A,B)~\mbox{for some}~y\in B\}.$$
$$B_0=\{y\in B: d(x,y)=d(A,B)~\mbox{for some}~x\in A\}.$$

\begin{definition}\cite{BS}
Let $(M,d)$ be a metric space and $A,B$ be two non-empty subsets of $M.$ Then $B$ is said to be approximatively compact with respect to $A$ if for every sequence $\{y_n\}$ of $B$ satisfying $d(x,y_n)\rightarrow d(x,B)$ as $n\rightarrow \infty$ for some $x\in A$ has a convergent subsequence.
\end{definition}

We need the following lemma from \cite{FE}.
\begin{lemma}\cite[\, proposition 3.3]{FE}\label{b}
Let $(A,B)$ be a nonempty and closed pair of subsets of a metric space $(X,d)$ such that $B$ is approximatively compact with respect to $A.$ Then $A_0$ is closed.
\end{lemma}

In \cite{AL}, Altun et al. proved the following best proximity point result.
\begin{theorem}\cite{AL}\label{a}
Let $A,B$ be nonempty and closed subsets of a complete metric space $(M,d)$ such that $A_0$ is nonempty and $B$ is approximatively compact with respect to $A.$ Let $T:A\rightarrow B$ be a $p$-proximal contraction such that $A_0\neq \phi$ and $T(A_0)\subseteq B_0.$ Then $T$ has an unique best proximity point. 
\end{theorem}

In \cite{GB}, Gabeleh and Markin showed that Theorem \ref{a} follows from the following fixed point theorem.
\begin{theorem}\cite{OP}\label{c}
Let $(M,d)$ be a complete metric space and $T:M\rightarrow M$ be a $p$-contraction mapping. Then $T$ has an unique fixed point and for any $x_0\in M,$ the Picard iteration sequence $\{T^{n}(x_0)\}$ converges to the fixed point of $T.$
\end{theorem}

Now we like to state our main result.
\begin{theorem}
If the $p$-proximal contraction constant $0<k<\frac{1}{3}$ then Theorem \ref{a} follows from the Banach contraction principle.
\end{theorem}

\begin{proof}
Let $x\in A_0.$ As $T(A_0)\subseteq B_0,$ so, $T(x)\in B_0.$ This implies there exists $y\in A_0$ such that $d(y,T(x))=d(A,B).$ Now, we will show that $y\in A_0$ is unique. Suppose there exists $y_1,y_2\in A_0$ such that $d(y_1,T(x))=d(A,B)$ and $d(y_2,T(x))=d(A,B).$ Since, $T:A\rightarrow B$ is a $p$-proximal contraction so we have,
$$d(y_1,y_2)\leq k\Big(d(x,x)+|d(y_1,x)-d(y_2,x)|\Big)\leq k d(y_1,y_2)$$
$$\Longrightarrow y_1=y_2.$$
Let $S_1:A_0\rightarrow A_0$ be defined by $S_1(x)=y.$ Now, we will show that $S$ is a contraction mapping. Let $x_1,x_2\in A_0.$ As $d(S_1(x_1),T(x_1))=d(A,B)$ and $d(S_1(x_2),T(x_2))=d(A,B)$ and $T$ is a $p$-proximal contraction so, we have,
$$d(S_1(x_1),S_1(x_2))\leq k\Big(d(x_1,x_2)+|d(S_1(x_1),x_1)-d(S_1(x_2),x_2)|\Big)$$ 
$$\Longrightarrow d(S_1(x_1),S_1(x_2))\leq k\Big(d(x_1,x_2)+ d(S_1(x_1),S_1(x_2))+d(x_1,x_2)|\Big)$$
$$\Longrightarrow d(S_1(x_1),S_1(x_2))\leq \frac{2k}{1-k} d(x_1,x_2).$$
Since $0<k<\frac{1}{3},$ so, $0<\frac{2k}{1-k}<1.$ This shows that $S:A_0\rightarrow A_0$ is a Banach contraction mapping. Now, from lemma \ref{b}, we can say $A_0$ is closed. So, $A_0$ is a complete metric space. Then, by Banach contraction principle the mapping $S_1$ has an unique fixed point $z\in A_0.$ Now, $d(z,T(z))=d(S(z),T(z))=d(A,B).$ This shows that $z$ is a best proximity point for $T.$ Uniqueness follows from the definition of $p$-proximal contraction. Also, we can conclude that for any $x_0\in A_0$ the sequence $\{S_1^{n}(x_0)\}$ will converge to the unique best proximity point of $T.$
\end{proof}

\section{conclusion}
The main motivation of the current paper is that if the $p$-proximal contraction constant $0<k<\frac{1}{3}$ then the best proximity point theorem by Altun \cite{AL} follows from the Banach contraction principle. If $\frac{1}{3}\leq k<1$ then the best proximity point theorem by Altun \cite{AL} follows from Theorem \ref{c} which is already shown by Gabeleh and Markin in \cite{GB}.

\end{document}